\documentclass[12pt]{amsart}
\usepackage{amscd}      
\usepackage{amssymb}
\usepackage{amsmath, amsthm, graphics}
\usepackage{xypic}      
\LaTeXdiagrams          
\usepackage[all,v2]{xy}
\xyoption{2cell} \UseAllTwocells \xyoption{frame} \CompileMatrices
\allowdisplaybreaks[3]

\usepackage{latexsym}
\usepackage{epsfig}
\usepackage{amsfonts}
\usepackage{enumerate}
\usepackage{times}

\newtheorem{prop}{Proposition}

\newtheorem{thm}[prop]{Theorem}

\newtheorem{lemma}{Lemma}

\theoremstyle{definition}

\theoremstyle{remark}

\theoremstyle{remark}

\newcommand{\Mbar}{\overline{\M}}
\newcommand{\proj}{\mathbb{P}}

\newcommand{\com}{\mathbb{C}}

\newcommand{\Z}{\mathcal{Z}}

\newcommand{\M}{\mathcal{M}}
\newcommand{\C}{\mathcal{C}}
\newcommand{\A}{\mathcal{A}}
\newcommand{\B}{\mathcal{B}}

\newcommand{\D}{\mathcal{D}}

\newcommand{\sO}{\mathcal{O}}

\newcommand{\bE}{\mathbb{E}}

\def\<{\left\langle}
\def\>{\right\rangle}

\begin{document}

\title{Abelian Hurwitz-Hodge integrals}
\author{P. Johnson}
\address{Department of Mathematics\\ University of Michigan\\
Ann Arbor, MI\\USA}
\email{pdjohnso@umich.edu}

\author{R. Pandharipande}
\address{Department of Mathematics\\Princeton University\\ Princeton, NJ\\USA}
\email{rahulp@math.princeton.edu}

\author{H.-H. Tseng}
\address{Department of Mathematics\\ University of Wisconsin-Madison\\ Madison\\ WI \\ USA}
\email{tseng@math.wisc.edu}

\date{\today}

\begin{abstract}
Hodge classes on the moduli space of admissible covers
with monodromy group $G$ are associated to irreducible
representations of $G$.
We evaluate all
linear Hodge integrals over moduli spaces of admissible covers with
abelian monodromy in terms of multiplication
in an associated wreath
group algebra. In case $G$ is cyclic and the representation
is faithful, the evaluation is in terms of double Hurwitz numbers.
In case $G$ is trivial, the formula specializes
to the well-known result of Ekedahl-Lando-Shapiro-Vainshtein for
linear Hodge integrals over the moduli space of curves in terms of
single Hurwitz numbers.
\end{abstract}

\maketitle

\tableofcontents

\setcounter{section}{-1}
\section{Introduction}
\subsection{Moduli of covers}
Let $\M_{g,n}$ be the moduli space of nonsingular, connected,
 genus $g$ curves over
$\mathbb{C}$
with
$n$ distinct points.
Let $G$ be a finite group.
Given an element $[C,p_1,\ldots,p_n] \in \M_{g,n}$, we will consider principal
$G$-bundles,
\begin{equation} \label{cccq}
\begin{CD}
G @>>> P \\
@. @VV\pi V \\
@. C\setminus\{p_1,\ldots,p_n\}\ \ ,
\end{CD}
\end{equation}
over the punctured curve. Denote the $G$-action on the fibers of
$\pi$ by
$$\tau: G \times P \rightarrow P.$$
The monodromy defined by a positively
oriented loop around the $i^{th}$ puncture
determines a conjugacy class $\gamma_i \in \text{Conj}(G)$.
Let $\mathbf{\gamma}=(\gamma_1,\ldots,\gamma_n)$
be the $n$-tuple
of monodromies.
The moduli space of covers $\A_{g,\gamma}(G)$ parameterizes
$G$-bundles \eqref{cccq} with the prescribed monodromy conditions.
There is a canonical morphism
$$\epsilon: \A_{g,\gamma}(G) \rightarrow \M_{g,n}$$
obtained from the base of the $G$-bundle. Both
$\A_{g,\gamma}(G)$ and $\M_{g,n}$ are nonsingular
Deligne-Mumford stacks.

A compactification $\A_{g,\gamma}(G)\subset \overline{\A}_{g,\gamma}(G)$
by {\em admissible
covers} was introduced by Harris and Mumford in \cite{HM}.
An admissible cover
$$[\pi,\tau]\in \overline{\A}_{g,\gamma}(G)$$
is a  degree $|G|$ finite  map of complete curves
$$ \pi: D \rightarrow (C,p_1,\ldots,p_n)$$
together with a $G$-action
$$ \tau: G \times D \rightarrow D$$
on the fibers of $\pi$
satisfying the following properties:
\begin{enumerate}
\item[(i)] $D$ is a possibly disconnected nodal curve,
\item[(ii)]$[C,p_1,\ldots,p_n]\in \overline{\M}_{g,n}$
is a stable curve,
\item[(iii)] $\pi$ maps the nonsingular points to nonsingular
points and nodes to nodes,
$$\pi(D^{ns}) \subset C^{ns}, \ \ \ \pi(D^{sing}) \subset C^{sing},$$

\item[(iv)] $[\pi,\tau]$ restricts to a principal $G$-bundle
 over the punctured nonsingular
locus
\begin{equation*}
\pi^{open}: D^{open} \rightarrow C^{ns}\setminus \{p_1,\ldots,p_n\}
\end{equation*}
with monodromy $\gamma$,
\item[(v)] distinct branches of  a node $\eta\in D^{sing}$
map to distinct branches
of $\pi(\eta)\in C^{sing}$
 with equal ramification orders over $\pi(\eta)$,

\item[(vi)]
the monodromies of the
$G$-bundle $\pi^{open}$ determined by the two branches of $C$
at $\eta \in C^{sing}$
lie in opposite conjugacy classes.
\end{enumerate}
Harris and Mumford originally considered only symmetric group $\Sigma_d$
monodromy, but the natural setting for the construction is
for all finite $G$.

An admissible cover may be alternatively viewed as a
principal $G$-bundle over the stack quotient{\footnote{$[D/G]$
differs from $C$ only by possible stack structure at the markings $p_i$
and the nodes. In both cases, the order of the isotropy group is
the order of the local monodromy in $G$.}}
 $[D/G]$ inducing
a stable map to the classifying space
\begin{equation}\label{cwwq}
f:[D/G] \rightarrow \B G.
\end{equation}
Then,
$\overline{\A}_{g,\gamma}(G)$ is simply
a moduli space of
stable maps \cite{AV,CR} {\footnote{We do not trivialize the marked
gerbes on the domain in the definition of $\overline{\M}_{g,\gamma}(\B G)
$.}},
$$\overline{\A}_{g,\gamma}(G) \stackrel{\sim}{=}
\overline{\M}_{g,\gamma}(\B G).$$ The deformation theory of
stable maps endows
$\overline{\A}_{g,\gamma}(G)$ with a canonical nonsingular
Deligne-Mumford stack structure. We take the stable maps
perspective here.

There are two flavors of such stable map theories. If the base
$C$ is required to be connected as above, we write
$\overline{\M}^\circ_{g,\gamma}(\B G)$. If disconnected bases
$C$ are allowed, we write $\overline{\M}^\bullet_{g,\gamma}(\B G)$.
In the disconnected case, the genus $g$ may be negative.
If the superscript is omitted, the connected case is
assumed.

Our results are restricted to abelian groups $G$.
Here, $\text{Conj(G)}$ is the set of elements of $G$.
Of
course, the cyclic groups $\mathbb{Z}_a$ will play the most important
role.
In case $G$ is trivial, there is no extra monodromy data, and
the moduli space of maps $\overline{\M}_{g,(0,\ldots,0)}(\B \mathbb{Z}_1)$
specializes to $\overline{\M}_{g,n}$.

\subsection{Hodge integrals}
Let $R$ be an irreducible $\com$-representation of $G$.
If $G$ is abelian, $R$ is a character
$$\phi^R: G \rightarrow \com^*.$$
By associating to each map
$[f]\in \overline{\M}_{g,\gamma}(G)$
presented as \eqref{cwwq} above
the $R$-summand of the $G$-representation $H^0(D,\omega_D)$,
we obtain a vector bundle
$$\mathbb{E}^R \rightarrow \overline{\M}_{g,\gamma}(\B G)\ .$$
The rank of $\mathbb{E}^R$ is locally constant and
determined by the orbifold Riemann-Roch formula
discussed in Section 1.
The {\em Hodge classes} on $\overline{\M}_{g,\gamma}(\B G)$ are
Chern classes of $\mathbb{E}^R$,
$$\lambda_i^R = c_i(\mathbb{E}^R) \in H^{2i}(\overline{\M}_{g,\gamma}(\B G),
\mathbb{Q}).$$

The $i^{th}$ cotangent line bundle $L_i$ on the moduli space
of curves has fiber
$$L_i|_{(C,p_1,\ldots,p_n)} = T^*_{p_i}(C).$$
Descendent classes on $\overline{\M}_{g,n}$ are defined by
$$\psi_i =c_1(L_i)  \in H^2(\overline{\M}_{g,n},\mathbb{Q}).$$
Descendent classes $\bar{\psi}_i$
on the space of stable maps are defined
by pull-back via the morphism
$$\epsilon: \overline{\M}_{g,\gamma}(\B G ) \rightarrow \overline{\M}_{g,n}$$
to the moduli space of curves,
$$\bar{\psi}_i = \epsilon^*(\psi_i) \in H^2(\overline{\M}_{g,\gamma}(\B G),
\mathbb{Q}).$$

The {\em Hodge integrals} over $\overline{\M}_{g,\gamma}(\B G)$
are the top intersection  products of the
classes
$\{\lambda_i^R\}_{R\in Irr(G)}$
and
$\{\bar{\psi}_j\}_{1\leq j \leq n}.$
Linear Hodge integrals
are of the form
$$\int_{\overline{\M}_{g,\gamma}(\B G)} \lambda_i^R \cdot \prod_{j=1}^n
\bar{\psi}_j^{m_j}.$$
The term {\em Hurwitz-Hodge integral} was used in \cite{BGP}
to emphasize the role of the covering spaces.

\subsection{Hurwitz numbers}
Let $g$ be a genus and let $\nu$ and $\mu$ be two (unordered) partitions
of $d\geq 1$.
Let $\ell(\nu)$ and $\ell(\mu)$ denote the lengths
of the respective partitions.
A Hurwitz cover of $\proj^1$ of
 genus $g$  with ramifications $\nu$
and $\mu$ over $0,\infty\in \proj^1$ is
a morphism
$$\pi: C \rightarrow \proj^1$$
satisfying the following properties:
\begin{enumerate}
\item[(i)] $C$ is a nonsingular, connected,
 genus $g$ curve,
\item[(ii)] the divisors $\pi^{-1} (0),\pi^{-1}(\infty)\subset C$ have profiles equal
            to the partitions $\nu$ and $\mu$ respectively,
\item[(iii)] the map $\pi$ is simply ramified over
$\mathbb{C}^*=\proj^1 \setminus\{0,
\infty\}$.
\end{enumerate}
By condition (ii), the degree of $\pi$ must be $d$.
Two covers
$$\pi: C \rightarrow \proj^1, \ \pi': C' \rightarrow \proj^1$$ are isomorphic if
there exists an isomorphism of curves $\phi: C \rightarrow C'$ satisfying
$\pi'\circ \phi= \pi$. Each  cover $\pi$ has an naturally
associated automorphism
group $\text{Aut}(\pi)$.

By the Riemann-Hurwitz formula,
the number of simple ramification points of $\pi$ over  $\mathbb{C}^*$
is $$r_g(\nu,\mu)=2g-2+\ell(\nu)+\ell(\mu).$$
Let $U_r\subset \mathbb{C}^*$ be a fixed set of $r_g(\nu,\mu)$
distinct points. The set of $r_g(\nu,\mu)^{th}$ roots of unity is the
standard choice.
The {\em double Hurwitz number}
$H_{g}(\nu,\mu)$ is a weighted count of the distinct
Hurwitz covers $\pi$ of genus $g$ with ramifications
$\nu$  and $\mu$ over $0,\infty\in \proj^1$
and simple ramification over $U_r$.
Each such cover is weighted by $1/|\text{Aut}(\pi)|$.
The count $H_g(\nu,\mu)$ does not depend upon the location of
the points of $U_r$.

There are two flavors of Hurwitz numbers. The connected case defined
above will be denoted $H^\circ_g(\nu,\mu)$. If $C$ is allowed to
be disconnected, the Hurwitz count is denoted $H^\bullet_g(\nu,\mu)$.
Again, the absence of a superscript indicates the connected theory.

Disconnected Hurwitz numbers are easily
 expressed as products in the center $\Z \Sigma_d$ of
 the group algebra of $\Sigma_d$,
\begin{equation} \label{hursym}
H^\bullet_g(\nu, \mu)=\frac{1}{d!}\big(
C_\nu T^{r_g(\nu, \mu)} C_\mu\big)_{[\text{Id}]}\ .
\end{equation}
Here, $C_\nu$ and $C_\mu$ are the sums in the group
algebra of all elements of $\Sigma_d$
with cycle types 
$\nu$ and $\mu$ respectively, and
$T$ is the sum of all transpositions.
The subscript denotes the coefficient of the identity
$[\text{Id}]$. 

Multiplication in $\Z\Sigma_d$ is diagonalized by the representation basis.
Hurwitz numbers can be written 
as sums over characters of $\Sigma_d$ and conveniently expressed
as matrix elements in the infinite wedge representation. The latter
formalism naturally connects Hurwitz numbers to integrable systems
\cite{Ok,OP1,P}.

\subsection{Formula for $\mathbb{Z}_a$}
The formula
for linear Hodge integrals is simplest in case the monodromy group is
$\mathbb{Z}_a$ and the representation $U$ is given by
$$\phi^U: \mathbb{Z}_a \rightarrow \mathbb{C}^*, \ \ \ \phi^U(1)= e^{\frac{2\pi i}{a}}.$$

Let $\gamma=(\gamma_1,\ldots, \gamma_n)$ be a vector{\footnote{
The length $n$ may be
taken to be 0 in which case $\gamma=\emptyset$.}} of {\em nontrivial}
elements of $\mathbb{Z}_a$,
$$\gamma_i \in \{ 1,\ldots, a-1\}.$$
Let $\mu$ be a partition of  $d\geq 1$ with parts $\mu_j$ and
length $\ell$,
$$\sum_{j=1}^\ell \mu_j = d.$$
Let $\gamma-\mu$ denote the vector of elements of $\mathbb{Z}_a$ defined by
$$\gamma-\mu = (\gamma_1,\ldots,\gamma_n, -\mu_1, \ldots, -\mu_\ell).$$
While the parts of $\mu$ are unordered, an ordering is chosen for
$\gamma-\mu$.
The vector $\gamma-\mu$ may contain trivial parts.
We will consider Hodge integrals over the moduli space
$\overline{\M}_{g,\gamma-\mu}(\B \mathbb{Z}_a)$.

For nonemptiness, the parity
condition
\begin{equation}\label{parityc}
d-\sum_{i=1}^n \gamma_i = 0  \mod a
\end{equation}
is required.
non-negativity,
$$
 d-\sum_{i=1}^n \gamma_i \geq 0,$$
and boundedness,
$$\forall i\neq j, \ \ \gamma_i+\gamma_j \leq a$$
will also be imposed.
If $\gamma=\emptyset$, non-negativity and boundedness are satisfied.

An automorphism of a partition is an element of the permutation group
preserving equal parts. Let $|\text{Aut}(\gamma)|$ and $|\text{Aut}(\mu)|$
denote the orders of the automorphism groups.{\footnote{Here, $\gamma$ is considered
as a partition by forgetting the ordering of the elements.}}
Let $\gamma_+$ be the partition of $d$ determined by
adjoining $\frac{d-\sum_{i=1}^n \gamma_i}{a}$ parts of size $a$,
$$\gamma_+= (\gamma_1,\ldots, \gamma_n, {a, \ldots,a}).$$
A calculation shows
$$ r_g(\gamma_+,\mu) = 2g-2 + n+
\ell +
\frac{d}{a} - \sum_{i=1}^n  \frac{\gamma_i}
{a}.$$

Let the monodromy group $\mathbb{Z}_a$ and representation $\phi^U$
be specified as above.
Our main result for linear $\mathbb{Z}_a$-Hodge integrals is the
following  formula.

\begin{thm} Let $\gamma=(\gamma_1,\ldots,\gamma_n)$ \label{vvv}
be nontrivial monodromies in $\mathbb{Z}_a$ satisfying
the parity, non-negativity, and boundedness conditions with
respect to the partition $\mu$. Then,
\begin{multline*}
H_g(\gamma_+,\mu) =\\
\frac{r_g(\gamma_+,\mu)!}{|\text{\em Aut}(\gamma)|\
|\text{\em Aut}(\mu)|} a^{1-g-
\sum_{i=1}^n \frac{\gamma_i}{a}+ \sum_{j=1}^\ell \< \frac{\mu_j}{a} \>}
\prod_{j=1}^\ell\frac{\mu_j^
{\left\lfloor\frac{\mu_j}{a}
\right\rfloor}}{\left\lfloor\frac{\mu_j}{a}\right\rfloor!}
\int_{\overline{\M}_{g,\gamma-\mu}(\B \mathbb{Z}_a)}
\frac{\sum_{i =0}^\infty
(-a)^i \lambda_i^{U}}{\prod_{j=1}^\ell (1-\mu_j \bar{\psi}_j)}\ .
\end{multline*}
\end{thm}

The integer and fractional parts of a rational number are denoted
in the above formula by
$$q = \lfloor q \rfloor + \< q\>, \ \ q\in \mathbb{Q}.$$
The cotangent lines in the denominator on the far right are associated
to the stack points of the stable map domain
corresponding to the parts of $\mu$.

Theorem \ref{vvv} is proven by virtual localization
on the moduli space of stable maps to the stack $\proj^1[a]$ with
$\mathbb{Z}_a$-structure at $0$ following the arguments of \cite{FP,GV1}.
The space of stable maps to $\proj^1[a]$
is discussed in Section 1, and the proof
is given in Section 2.
The formula is easily seen to determine {\em all}
linear $\mathbb{Z}_a$-Hodge integrals
with respect to ${U}$  in terms of double Hurwitz numbers.
In fact, the set of evaluations with $\gamma = \emptyset$ is sufficient.
Conversely, every double Hurwitz number is realized for $a$ sufficiently large.

For the disconnected formula, we assume $\gamma=\emptyset$
and the parity condition $d = 0 \ (\text{mod} \ a)$.{\footnote{If $\gamma\neq \emptyset$, 
the non-negativity condition may satisfied globally
 but be violated on connected components.}} Then,
Theorem \ref{vvv} holds in exactly the same form,
\begin{equation}\label{dderr}
H^\bullet_g(\emptyset_+,\mu) =
\frac{r_g(\emptyset_+,\mu)!}{
|\text{Aut}(\mu)|} a^{1-g
+ \sum_{j=1}^\ell \< \frac{\mu_j}{a} \>}
\prod_{j=1}^\ell\frac{\mu_j^
{\left\lfloor\frac{\mu_j}{a}
\right\rfloor}}{\left\lfloor\frac{\mu_j}{a}\right\rfloor!}
\int_{\overline{\M}^\bullet_{g,-\mu}(\B \mathbb{Z}_a)}
\frac{\sum_{i =0}^\infty
(-a)^i \lambda_i^{U}}{\prod_{j=1}^\ell (1-\mu_j \bar{\psi}_j)}\ .
\end{equation}

The ELSV formula \cite{ELSV} for
linear Hodge integrals on the moduli space of curves
arises from the $a=1$ specialization of Theorem \ref{vvv},
\begin{equation*}
H_g(\mu) =\\
\frac{(2g-2+d+\ell)!}{|\text{Aut}(\mu)|}
\prod_{j=1}^\ell\frac{\mu_j^{\mu_j}}
{\mu_j!}
\int_{\overline{\M}_{g,\ell}}
\frac{\sum_{i=0}^g
(-1)^i \lambda_i}{\prod_{j=1}^\ell (1-\mu_j {\psi}_j)}\ .
\end{equation*}
For $a=1$, we must have $\gamma=\emptyset$.

The conditions $\gamma$ allow for greater freedom in the $a> 1$
case.
For example, the proof of Theorem \ref{vvv} yields a remarkable
vanishing property. The monodromy conditions $\gamma$ satisfy
{negativity} if
$$d-\sum_{i=1}^n \gamma_i <0$$
and {strong negativity} if
$$d-n - \frac{d-\sum_{i=1}^n \gamma_i}{a} <0.$$
Strong negativity is easily seen to imply negativity.

\begin{thm}\label{vvvv} Let $\gamma=(\gamma_1,\ldots,\gamma_n)$
be nontrivial
monodromies in $\mathbb{Z}_a$ satisfying
the parity condition with respect to the partition $\mu$.
In addition, let $\gamma$ satisfy at least one of the
following two conditions:
\begin{enumerate}
\item[(i)]  negativity and boundedness, or \\
\item[(ii)] strong negativity.
\end{enumerate}
Then, a vanishing results for Hurwitz-Hodge integrals holds:
\begin{equation*}
\int_{\overline{\M}_{g,\gamma-\mu}(\B \mathbb{Z}_a)}
\frac{\sum_{i=0}^\infty
(-a)^i \lambda_i^{U}}{\prod_{j=1}^\ell (1-\mu_j \bar{\psi}_j)} =0.
\end{equation*}
\end{thm}

A few examples of Theorems \ref{vvv} and \ref{vvvv} where alternative
approaches to the integrals are available are presented in
Section 3.

\subsection{Abelian $G$} 
Since any faithful representation $R$ of $\mathbb{Z}_a$ differs from $U$ 
by an automorphism of $\mathbb{Z}_a$, Theorem \ref{vvv} determines linear
Hodge integrals with respect to $R$.
Representations of $\mathbb{Z}_a$ with kernels require an additional analysis.

Let $G$ be an abelian group with group law written additively.
Consider an irreducible representation $R$,
$$\phi^R: G \rightarrow \com^*,$$  with associated exact sequence
\begin{equation} \label{exactsequence}
0 \rightarrow K \rightarrow G \stackrel{\phi^R}{\rightarrow}
 \text{Im}(\phi^R) \stackrel{\sim}{=} \mathbb{Z}_a \rightarrow0.
 \end{equation}

The homomorphism $\phi^R$ induces a canonical morphism
$$\rho: \overline{\M}_{g,\gamma}(\B G) \rightarrow
\overline{\M}_{g,\phi^R(\gamma)}(\B \mathbb{Z}_a).$$
The morphism $\rho$ satisfies
$$\rho^*(\lambda_i^U) = \lambda_i^R$$ and has the same degree over each component of $\overline{\M}_{g,\phi^R(\gamma)}(\B \mathbb{Z}_a)$.
Therefore, linear Hodge integrals with respect to
$R$ can be calculated by multiplying the formula of Theorem
\ref{vvv}
by the degree of $\rho$.

In Section 4, the solution for arbitrary $G$ and $R$ is cast
in a more appealing way.  When $$\phi^R(\gamma)=-\mu\in\mathbb{Z}_a,$$
 Hodge integrals of the form
$$\int_{\overline{\M}_{g,\gamma}(\B G)}
\frac{\sum_{i =0}^\infty
(-a)^i \lambda_i^{R}}{\prod_{j=1}^\ell (1-\mu_j \bar{\psi}_j)}$$
are expressed in terms of Hurwitz numbers for $K_d$, the wreath product of $K$ with the symmetric group $\Sigma_d$.  
Since the infinite wedge formalism for $\Sigma_d$ extends to a Fock space formalism for the wreath product $K_d$, 
there is again a connection to integrable systems \cite{QW}.

Conjugacy classes in $K_d$ are indexed by $\text{Conj}(K)$-weighted partitions of $d$,
$$ \overline{\mu}=\{(\mu_1,\kappa_1), \dots, (\mu_{\ell(\mu)}, \kappa_{\ell(\mu)})\}. $$
Here, $\mu$ is a partition of $d$ with parts $\mu_j$, the weights
 $\kappa_i\in\text{Conj}(K)$ are conjugacy classes 
in $K$, and $\overline{\mu}$
is an unordered set of pairs.  
Let $\text{Aut}(\overline{\mu})$ denote
the automorphism group of $\overline{\mu}$.
Let  $C_{\overline{\mu}}\in \Z K_d$ be the element of the
group algebra associated to the conjugacy class $\overline{\mu}$.
The transposition element
$T\in \Z K_d$ is associated to conjugacy class of $K_d$ indexed by
$$\overline{\tau}=\{(2,0),(1,0),\ldots,(1,0)\}$$
where all the $\text{Conj}(K)$-weights are 0.

The wreath product $K_d$ has a forgetful map to $\Sigma_d$ which sends
 elements of cycle type $\overline{\mu}$ to elements of type $\mu$.  
The {\em $K_d$-Hurwitz number}
$H_{g,K}(\overline{\nu},\overline{\mu})$
counts the degree $d|K|$-fold covers of $\proj^1$ with monodromy in $K_d$ 
given by $\overline{\nu}$ and $\overline{\mu}$ at $0,
\infty\in \proj^1$ and $\overline{\tau}$ at all the points of
$$U_{r_{g}(\nu,\mu)}\subset \proj^1.$$
Since $K\subset K_d$ is contained in the center,
any such cover has a canonical $K$-action which defines a $K$-bundle 
over a punctured Hurwitz cover 
counted by $H_g(\nu,\mu)$.  
The connectivity requirement we place on covers 
counted by $H_{g,K}(\overline{\nu},
\overline{\mu})$
is {\em not} that the $d|K|$-fold cover is connected, but only that the
associated Hurwitz $d$-fold cover is connected.  
Similarly, $g$ is the genus of the $d$-fold cover.

The natural extension of formula (\ref{hursym}) for disconnected
Hurwitz covers for the wreath product $K_d$ is
$$
H^\bullet_{g, K}(\overline{\nu}, \overline{\mu})=
\frac{1}{|K_d|}\big(
C_{\overline{\nu}}T^{r_g(\nu, \mu)} 
C_{\overline{\mu}}\big)_{[\text{Id}]} \ ,
$$
where the product on the right takes place in the group algebra of $K_d$.

Select an element $x\in G$ with $\phi^R(x)=1$.  
Let  $k=ax\in K$.  Denote by $-\overline{\mu}$ the 
$\ell({\mu})$-tuple of elements of 
$G$ defined by:
$$-\overline{\mu}
=(\kappa_1-\mu_1 x,\kappa_2-\mu_2 x,\dots,\kappa_{\ell(\mu)} -\mu_{\ell(\mu)} x).$$
Although the parts of $\overline{\mu}$ are unordered, 
an ordering is chosen for $-\overline{\mu}$.  
The parity condition is now
$$\sum_{j=1}^\ell \kappa_j -\mu_j x = 0 \in G.$$
Denote by
 $\emptyset_{+}(k)$ the conjugacy class given by
$$\emptyset_{+}(k)=\{\underbrace{ (a, -k),\dots, (a,-k)}_{\text{$d/a$ times}}\}.$$

\begin{thm} For weighted-partitions $\overline{\mu}$ satisfying the parity condition,
\label{vvvvv}
$$H_{g,K}(\emptyset_+(k), \overline{\mu})
=\frac{r_g(\emptyset_+,\mu)!}{\
|\text{\em Aut}(\overline{\mu})|} a^{1-g+ \sum_{j=1}^\ell \< \frac{\mu_j}{a} \>}
\prod_{j=1}^\ell\frac{\mu_j^
{\left\lfloor\frac{\mu_j}{a}
\right\rfloor}}{\left\lfloor\frac{\mu_j}{a}\right\rfloor!}
\int_{\overline{\M}_{g,-\overline{\mu}}(\B G)}
\frac{\sum_{i =0}^\infty
(-a)^i \lambda_i^{R}}{\prod_{j=1}^\ell (1-\mu_j \bar{\psi}_j)}\ .$$
\end{thm}

Theorem \ref{vvvvv} determines all linear Hurwitz-Hodge integrals for $G$
and holds in exactly the same form for
the disconnected theories $H_{g,K}^\bullet(\emptyset_+(k),
\overline{\mu})$ and 
$\overline{\M}^\bullet_{g,-\overline{\mu}}(\B G)$.

\subsection{Future directions}
The ELSV formula has two immediate applications in Gromov-Witten theory.
The first is the determination of descendent integrals over
$\overline{\M}_{g,n}$ via asymptotics to remove the Hodge classes
\cite{KL,OP1}. The second is the exact evaluation of the
vertex integrals in the localization formula for $\proj^1$
in \cite{OP2,OP3}. The latter requires the Hodge classes.

Since $\epsilon:\overline{\M}_{g,\gamma}(\B G) \rightarrow
\overline{\M}_{g,n}$ is a
finite map, a geometric approach to the descendent
integrals is not strictly necessary \cite{JK}. However,
for the calculation of the Gromov-Witten theory of target curves with
orbifold structure \cite{PJ}, Theorem \ref{vvvvv} 
is essential.
The results
may be viewed as a first step for orbifolds
along the successful line of exact Hodge integral formulas
which have culminated in the topological and equivariant
vertices in ordinary Gromov-Witten theory.

Hurwitz-Hodge integrals can be viewed 
as pairings of tautological classes
$$\epsilon_*(\lambda_i^R) \in H^{2i}(\overline\M_{g,n},\mathbb{Q})$$
against the descendents $\psi_i$. 
Given an
action
$$\alpha:G \times \{1,\ldots,k\} \rightarrow \{1,\ldots,k\}$$
on a set with $k$ elements, there is a second map to the
moduli space of curves.
Let
$$ \C \rightarrow
\overline{\M}_{g,\gamma}(\B G), \ \
\mathcal{D} \rightarrow \C
$$
be the universal domain curve and the
universal $G$-bundle respectively.
A second universal curve
$$\mathcal{D}^\alpha = \mathcal{D}\times _G \{1,\ldots,k\} \rightarrow
\overline{\M}_{g,\gamma}(\B G)$$
is obtained by the mixing construction. We obtain
$${\epsilon}^\alpha:\overline{\M}_{g,\gamma}(\B G) \rightarrow
\overline{\M}_{g^\alpha,n^\alpha},$$
where $g^\alpha$ and $n^\alpha$ are the genus and the number
of distinguished
sections{\footnote{We
suppress the ordering issues  here.}}
of the universal curve $\mathcal{D}^\alpha$.
Two questions immediately arise:
\begin{enumerate}
\item[(i)] Do the classes ${\epsilon}^\alpha_*(\lambda_i^R)$ lie in the
tautological ring of
$\overline{\M}_{g^\alpha,n^\alpha}$?
\item[(ii)] Do the pairings of 
${\epsilon}^\alpha_*(\lambda_i^R)$ against the
descendents of $\overline{\M}_{g^\alpha,n^\alpha}$
admit simple evaluations?
\end{enumerate}
The answer to (i) is known \cite{GPC} to be false for $g=1$,
but may be true for $g=0$.
See \cite{FPJ} for positive results related to (i) for
the standard action
of the symmetric group $\Sigma_k$ in the $g=0$ case.

\subsection{Acknowledgments}
We thank J. Bryan, R. Cavalieri, 
T. Graber, C. Faber, D. Maulik, A. Okounkov, Y. Ruan, and
R. Vakil
for related conversations.

P.J. was partially supported by RTG grant DMS-0602191 at the University
of Michigan.
R.P. was partially supported
by DMS-0500187. H.-H. T.
thanks the Institut Mittag-Leffler for hospitality and support during
a visit in Spring 2007.
The paper was furthered at a
lunch in Kyoto while the last two authors were visiting RIMS in
January 2008. Section \ref{bnf} was added after discussions at the
Banff workshop on
{\em Recent progress on the moduli of curves} in March 2008.

\section{Stable relative maps} \label{sm}
\subsection{Definitions}
For $a\geq 1$, let
$\proj^1[a]$ be the projective line
 with a single stack point of order $a$ at $0$. Let
$$\<\zeta_a\>\subset \com^*, \ \ \ \zeta_a = e^{\frac{2\pi i}{a}}  $$
be the group of $a^{th}$-roots of unity.
Locally at 0, $\proj^1[a]$ is
the quotient stack
$\com/ \<\zeta_a\>$.
Alternatively, $\proj^1[a]$ is the  $a^{th}$-root stack of
  $\proj^1$
along the divisor $0$.

Let $\Mbar_{g,\gamma}(\proj^1[a],\mu)$ be the stack of stable
relative maps to $(\proj^1[a], \infty)$ where
$\gamma= (\gamma_1, \ldots, \gamma_n)$
is a vector of nontrivial elements
 $$1\leq \gamma_i \leq  a-1, \ \ \gamma_i \in \mathbb{Z}_a,$$
and  $\mu$ is a partition of $d\geq 1$ with parts $\mu_j$ and length $\ell$.
The moduli space parametrizes maps
$$[\ f:(C,p_1,\ldots,p_n) \rightarrow \proj^1[a]\ ] \in
\Mbar_{g,\gamma}(\proj^1[a],\mu)$$
for which
\begin{enumerate}
\item[(i)] the domain $C$ is a nodal curve of genus $g$ with stack structure
at $p_i$ determined by $\gamma_i$,
\item[(ii)]
relative conditions over $\infty\in \proj^1[a]$
are given by the partition $\mu$.
\end{enumerate}
The isotropy group of $p_i\in C$ is the subgroup of $\mathbb{Z}_a$
generated by $\gamma_i$. Let $a_i$ denote the order of $\gamma_i$.
The domain $C$, called
a {\em twisted curve}, may have additional stack
structure at the nodes, see \cite{AV}.

We recall the Riemann-Roch formula for twisted curves.{\footnote{See
Theorem 7.2.1 of \cite{AGV} for precisely our situation.}}
Let $C$ be a twisted curve whose nonsingular
 stack points are $p_1,..., p_n$ with cyclic isotropy groups
$I_1,\ldots,I_n$. The group $I_i$ is
identified with the $a_i^{th}$-roots of unity via the action on $T_{p_i}C$,
$$I_i \stackrel{\sim}{\rightarrow} \<\zeta_{a_i}\>\subset \com^*, \ \
\zeta_{a_i} = e^{\frac{2\pi i}{a_i}}.$$
 Let $E$ be a locally free sheaf over the stack $C$.
 Then, $I_i$ acts on the restriction $E|_{p_i}$. Let
$$E|_{p_i}=\bigoplus_{0\leq s\leq a_i-1}V_s^{\oplus e_s}$$
be the direct sum decomposition,
 where
$V_s$ is the irreducible representation of $\mathbb{Z}_{a_i}$ associated to the character
$$\phi^s:I_i\to \com^*,\ \ \ \phi^s(\zeta_{a_i})= \zeta_{a_i}^s .$$
 The {\em age} of $E$ at $p_i$ is defined by
 $$\text{age}_{p_i}(E)=\sum_{0\leq s\leq a_i-1} e_s \frac{s}{a_i}\ .$$
The Riemann-Roch formula for twisted curves is given as follows:
\begin{equation}\label{rr_twisted_curve}
\chi(C, E)=\text{rk}(E) (1-g)+\text{deg}(E)-\sum_{i=1}^n\text{age}_{p_i}(E).
\end{equation}

The virtual dimension of $\Mbar_{g,\gamma}(\proj^1[a],\mu)$ is calculated by the Riemann-Roch formula
\eqref{rr_twisted_curve}.
 Let
$$[ \ f:(C,p_1,\ldots,p_n)\to \proj^1[a]\ ]\in \Mbar_{g,\gamma}(\proj^1[a],\mu).$$
Certainly, $\text{deg}\left( f^*T_{\proj^1[a]}(-\infty)\right)=d/a$.
By the quotient presentation of $\proj^1[a]$, the character of $f^*T_{0,\proj^1[a]}$  at $p_i$
is
$$\zeta_{a_i} \mapsto \zeta_{a_i}^{\frac{\gamma_i a_i }{a}}= \zeta_a^{\gamma_i}.$$
Therefore,
$\text{age}_{p_i}\left(f^*T_{\proj^1[a]}(-\infty)\right)= \frac{\gamma_i}{a}$ and
\begin{equation*}
\begin{split}
\text{vdim}\,\Mbar_{g,\gamma}(\proj^1[a],\mu)&=3g-3+n+\ell+\chi(C, f^*T_{\proj^1[a]}(-\infty))\\
&=3g-3+n+\ell+ 1-g+\frac{d}{a}-\sum_{i=1}^n \frac{\gamma_i}{a}\\
&=2g-2+n+\ell+\frac{d}{a}-\sum_{i=1}^n\frac{\gamma_i}{a}\ .
\end{split}
\end{equation*}

To simplify notation, let $r$ denote the above virtual dimension.
Since $r$ must be an integer,
$\Mbar_{g,\gamma}(\proj^1[a],\mu)$ is empty unless the parity
condition $d = \sum_{i=1}^n \gamma_i \,(\text{mod }a)$ holds.

\subsection{Hurwitz numbers} \label{xde}
We now impose the
non-negativity condition,
$$d - \sum_{i=1}^n \gamma_i \geq 0.$$
Let $H_{g,a}(\gamma,\mu)$ denote the weighted count of
degree $d$ representable maps from nonsingular, connected,
genus $g$  twisted curves with
stack points of type $\gamma$ to $\proj^1[a]$ with profile $\mu$ over
$\infty$ and simple ramification
 over $r$ fixed points in $\proj^1[a]\setminus
\{0,\infty\}$.

\begin{lemma} \label{ozoz}
$H_{g,a}(\gamma,\mu)$ is well-defined and
equal to $|\text{\em Aut}(\gamma)| \cdot H_g(\gamma_+,\mu)$.
\end{lemma}

Given a stack map $[f:C\rightarrow \proj^1[a]]
\in \M_{g,\gamma}(\proj^1[a],\mu)$
satisfying the simple ramification condition over the $r$ points,
the associated coarse  map
$$f^c:C^c \rightarrow \proj^1$$
is a
usual Hurwitz covering counted by $H_g(\gamma_+,\mu)$.
The representability condition implies the point $p_i$ has ramification
profile $\gamma_i$ over 0 for the coarse map.
Conversely, we have the following result.

\begin{lemma}
Let $C^c$ be a nonsingular curve and let
$f^c: C^c\to \proj^1$ be a nonconstant map.
Then, there is a unique (up to isomorphism) twisted curve $(C,p_1,\ldots,p_m)$
and a representable morphism $f:C\to \proj^1[a]$ whose
induced map between coarse curves is $f^c$.
\end{lemma}

\begin{proof}
 Since the natural map $\proj^1[a]\to \proj^1$ is an isomorphism
over $\proj^1[a]\setminus [0/\mathbb{Z}_a]$, we may consider the composite
$$C^c\setminus  (f^c)^{-1}(0)\overset{f^c}{\longrightarrow}\proj^1
\setminus \{0\}\overset{\sim}{\longrightarrow} \proj^1[a]\setminus
\{[0/\mathbb{Z}_a]\}\subset \proj^1[a].$$ The Lemma follows by applying
Lemma 7.2.6 of \cite{AV}.
\end{proof}

To proceed, we need to identify the ramification profile of $f^c$ over $0$.
Since $\proj^1[a]$ is a root stack, we may use classification results
on maps to root stacks proven in \cite{C}.
According to  Theorem 3.3.6 of \cite{C}, maps considered
in our stack Hurwitz problem are in bijective
correspondence with maps $f^c: C^c\to \proj^1$ from a
coarse curve $C^c$ satisfying
\begin{equation}\label{ramification_type}
(f^c)^*[0]=\sum_{i=1}^n \gamma_i [\bar{p}_i]+ a D,
\end{equation}
where $\bar{p}_1,...,\bar{p}_n\in C^c$ are distinct points
 and $D\subset C^c$ is a divisor consisting of
$\frac{d-\sum_{i=1}^n \gamma_i}{a}$ additional distinct points.

The proof of Lemma \ref{ozoz} is complete. The factor $|\text{Aut}(\gamma)|$
occurs since the stack points of $C$ are labelled while the
corresponding ramification points on the Hurwitz covers enumerated
by $H_g(\gamma_+,\mu)$ are not. $\Box$

\subsection{Branch maps}
There exists a basic branch morphism for stable maps,
$$\text{br}: \Mbar_{g}(\proj^1,\mu)\to \text{Sym}^{2g-2+d+\ell}(\proj^1),$$
constructed in \cite{FP}. By composing with the coarsening
 map, we obtain
$$\text{br}: \Mbar_{g,\gamma}(\proj^1[a],\mu)\to \text{Sym}^{2g-2+d+\ell}
(\proj^1).$$ To proceed, we impose the boundedness condition,
$$\forall i \neq j , \  \ \gamma_i + \gamma_j \leq a.$$

\begin{lemma}\label{HHH}
If the parity, non-negativity, and boundedness conditions are satisfied,
$$\text{\em Im}(\text{\em br})
\subset \left(d-n-\frac{d-\sum_{i=1}^n \gamma_i}{a}\right)[0]
+ \text{\em Sym}^r(\proj^1) \subset
\text{\em Sym}^{2g-2+d+\ell}(\proj^1).$$
\end{lemma}

\begin{proof} Let $f:C\rightarrow \proj^1[a]$ be
a Hurwitz cover counted by $H_{g,a}(\gamma,\mu)$.
The expression $$E=d-n-\frac{d-\sum_{i=1}^n \gamma_i}{a}$$
is the order of $[0]$ in $\text{br}([f])$.
The claim of the Lemma is simply that the
minimum order of $[0]$ in $\text{br}(f)$ is achieved at
such Hurwitz covers $f$.

The proof requires checking all possible degenerations
of $f$ over 0. If the stack points $p_1,\ldots,p_n$
do not bubble off the domain, the claim follows easily as
in the coarse case.
We leave the details to the reader.

A more interesting calculus is encountered if
a subset of stack points $p_1,\ldots, p_l$ bubbles off the
domain together over $[0/\mathbb{Z}_a]\in \proj^1[a]$.
We do the analysis for a single bubble.
We can assume the bubble is of genus 0 since higher
genus increases the branching order.
The multi-bubble calculation is identical.

The genus 0
bubble is attached to the rest of the curve in $m$ stack points
of type
$$\delta_1,\ldots,\delta_m \in \mathbb{Z}_a, \ \ 1\leq \delta_j \leq a$$
on the noncollapsed side.
The parity condition
\begin{equation}\label{nnhh}
\sum_{i=1}^l \gamma_i - \sum_{j=1}^m \delta_j = ka
\end{equation}
must be satisfied with $k \in \mathbb{Z}$.

The branch contribution over
0 of the bubbled map is at least
$$E'=\sum_{i=l+1}^n (\gamma_i-1) + \sum_{j=1}^m (\delta_j-1) + 2m-2 +
\frac{d- \sum_{i=l+1}^n \gamma_i -\sum_{j=1}^m \delta_j}{a}(a-1).$$
All the terms on the right are obtained from the ramifications
on the noncollapsed side except for the $2m$ from the
$m$ nodes of the bubble and the $-2$ from the bubble itself, see \cite{FP}.
Rewriting using the parity condition \eqref{nnhh}, we find
$$E'= E +l +m-2-k.$$
By connectedness and bubble stability, we have
 $$m\geq 1, \ \ l+m\geq 3.$$
If $k\leq 0$, we conclude $E'>E$.
If $k\geq 0$, then $k\leq l-2$ by the boundedness condition
and the positivity of $\delta_1$. Again, $E'>E$.
\end{proof}

By Lemma \ref{HHH}, we may view the branch map with restricted image,
$$\text{br}_0: \overline{\M}_{g,\gamma}(\proj^1[a],\mu)
 \rightarrow \text{Sym}^r(\proj^1).$$
The proof of Lemma \ref{HHH} shows the maps
$f:C\rightarrow \proj^1[a]$ satisfying
$[0]\notin \text{br}_0(f)$
have no contraction over $0$ and coarse profile exactly $\gamma_+$.
The usual nonsingularity and Bertini arguments \cite{FP}
then imply the following result.

\begin{lemma} \label{gew}
If the parity, non-negativity, and boundedness conditions are satisfied,
$$H_{g,a}(\gamma,\mu)=\int_{[\overline{\M}_{g,\gamma}(\proj^1[a],\mu)]^{vir}}
\text{\em br}_0^*(H^r),$$
where $H\in H^2(\text{\em Sym}^r(\proj^1),\mathbb{Q})$ is the
hyperplane class.
\end{lemma}

\section{Localization}
\subsection{Fixed loci}
\label{flo}
The standard $\com^*$-action on $\proj^1$, defined
by $\xi\cdot [z_0,z_1]=[ z_0,\xi z_1]$, lifts canonically to
$\com^*$-actions on $\proj^1[a]$ and
$\overline{\M}_{g,\gamma}(\proj^1[a],\mu)$ .
We will evaluate the integral
\begin{equation}\label{lq23}
\int_{[\overline{\M}_{g,\gamma}(\proj^1[a],\mu)]^{vir}}
\text{br}_0^*(H^r)
\end{equation}
by virtual localization
for relative maps \cite{GP,GV2} following \cite{FP,GV1}.
We assume the parity, non-negativity, and boundedness conditions.

The first step is to define a lift of the  $\com^*$-action to the
integrand. Certainly the $\com^*$-action lifts canonically to
$\text{Sym}^r(\proj^1)$. A lift of $H^r$ can be defined by choosing
the $\com^*$-fixed point  $r[0]\in \text{Sym}^r(\proj^1)$.
The tangent weights at
$[0/\mathbb{Z}_a],\infty\in \proj^1[a]$ are
 $\frac{t}{a}$ and $-t$ respectively.
The equivariant Euler class of the normal bundle to $r[0]$ in
$\text{Sym}^r(\proj^1)$
has weight
$r!t^r$.

The second step is to identify the $\com^*$-fixed locus
$\overline{\M}_{g,\gamma}(\proj^1[a],\mu)^{\com^*} \subset
\overline{\M}_{g,\gamma}(\proj^1[a],\mu)$.
The components of  the $\com^*$-fixed locus
lie over the $r+1$ points of $\text{Sym}^r(\proj^1)^{\com^*}$.
By our lifting of $H^r$, we need only consider
$$\overline{\M}^{\com^*}_0 =
\overline{\M}_{g,\gamma}(\proj^1[a],\mu)^{\com^*}
\cap \text{br}_0^{-1}(r[0]).$$
Because of the strong restriction on the branching, the maps
$$[f:C \rightarrow \proj^1[a]]\in \overline{\M}_0^{\com^*}$$
have a very simple structure:
\begin{enumerate}
\item[(i)]
$C=C_0\cup\coprod_{j=1}^\ell C_j$, \\
\item[(ii)]
$f|_{C_0}$ is a constant map from a genus $g$ curve
to $[0/\mathbb{Z}_a]\in \proj^1[a]$, \\
\item[(iii)]
the coarse map $f^c|_{C_j}:C^c_j\to \proj^1$ is
a $\com^*$-fixed Galois cover of degree $\mu_j$ for $j>0$, \\
\item[(iv)]
$C_0$ meets $C_j$ at a node $q_j$.
\end{enumerate}
The stack structure at $q_j\in  C_j$ is easily
  determined using the relationship between stack Hurwitz covers
of $\proj^1[a]$ and ordinary Hurwitz covers of $\proj^1$
discussed in Section \ref{xde}.
The  stack structure at $q_j\in C_j$ is
of type $\mu_j \in \mathbb{Z}_a$.
The stack structure at  $q_j\in C_0$ where  $C_j$
is attached is of the {\em opposite} type  $-\mu_j\in \mathbb{Z}_a$.
The map
$$f|_{C_0} : (C,p_1,\ldots,p_n, q_1, \ldots, q_\ell) \rightarrow [0/\mathbb{Z}_a]$$
is an element of $\overline{\M}_{g,\gamma-\mu}(\B \mathbb{Z}_a)$.

The $\com^*$-fixed locus may be identified
with a quotient of a fibered product,
$$\overline{\M}_0^{\com^*}\stackrel{\sim}{=}
\Big( \overline{\M}_{g,\gamma-\mu}(\B \mathbb{Z}_a)
\times_{(\bar{I}\B\mathbb{Z}_a^{\ell})}
P_1\times...\times P_{\ell}\Big)_{/\text{Aut}(\mu)},$$
where
$\bar{I}\B \mathbb{Z}_a$ is the rigidified inertia stack of $\B\mathbb{Z}_a$
and
$P_j$ is the moduli stack of $\com^*$-fixed Galois covers of
degree $\mu_j$.
By the standard multiplicity obtained from gluing stack
$\mathbb{Z}_a$-bundles,
the projection
\begin{equation}\label{ffor}
\overline{\M}_0^{\com^*} \rightarrow
\Big(\overline{\M}_{g,\gamma-\mu}(\B \mathbb{Z}_a)
\times
P_1\times...\times P_{\ell}\Big)_{/\text{Aut}(\mu)}
\end{equation}
has degree $\prod_{j=1}^\ell \frac{a}{b_j}$ where $b_j$ is the order
of $\mu_j\in \mathbb{Z}_a$.

Fortunately, the residue integral over $\overline{\M}_0^{\com^*}$
 in the virtual localization
formula for \eqref{lq23} is pulled-back via \eqref{ffor}.
Instead of integrating over $\overline{\M}_0^{\com^*}$,  we will
integrate over
$$\widetilde{\M}^{\com^*}_0=\overline{\M}_{g,\gamma-\mu}(\B \mathbb{Z}_a)
\times
P_1\times...\times P_{\ell}$$
and multiply by
$$\frac{1}{|\text{Aut}(\mu)|} \prod_{j=1}^\ell \frac{a}{b_j}.$$

\subsection{Virtual normal bundle}
The virtual localization formula for \eqref{lq23} with our
choice of equivariant lifts takes the following form:
\begin{equation}\label{l23}
\int_{[\overline{\M}_{g,\gamma}(\proj^1[a],\mu)]^{vir}}
\text{br}_0^*(H^r)  = \frac{1}{|\text{Aut}(\mu)|}
\prod_{j=1}^\ell \frac{a}{b_j}
\int_{\widetilde{\M}^{\com^*}_0} \frac{r!\ t^r}{e(\text{Norm}^{vir})}.
\end{equation}
The equivariant Euler class of
the virtual normal bundle is
\begin{equation}\label{nnrrmm}
\frac{1}{e(\text{Norm}^{vir})} =
\frac{e(H^1(C,f^*T_{\proj^1[a]}(-\infty)))}
     {e(H^0(C,f^*T_{\proj^1[a]}(-\infty)))}
\frac{1}{\prod_{j=1}^\ell e(N_j)},
\end{equation}
see \cite{GP}. The last product is over the nodes of $C$,
and $N_j$ is the equivariant line bundle associated to the
smoothing of $q_j$. The terms in \eqref{nnrrmm} are
computed  via the normalization sequence of the domain
$C$. The various contributions over the components
$C_0,C_1, \ldots, C_\ell$ are computed separately.

First consider the collapsed component $C_0$. The space
$H^0(C_0, f|_{C_0}^*T_{\proj^1[a]}(-\infty))$
is identified with the subspace of $T_{\proj^1[a]}(-\infty)
|_{[0/\mathbb{Z}_a]}$ consisting of vectors invariant under the
action of the image of the monodromy representation
$\pi_1^{orb}(C_0)\to \mathbb{Z}_a$.
Therefore, $H^0$ vanishes unless the monodromy representation is
trivial, in which case $H^0$ is 1-dimensional with weight $\frac{t}{a}$.

The trivial monodromy representation $\pi_1^{orb}(C_0)\to \mathbb{Z}_a$
is possible only if
$$\gamma=\emptyset \ \ \text{ and } \ \ \forall j, \
 \mu_j = 0\ \ \text{mod} \  a\ .$$
Even then, the locus with trivial monodromy is just a
component{\footnote{If $g>0$, there will typically be other
components as well.}} of
$\overline{\M}_{g,(0,\ldots,0)}(\B \mathbb{Z}_a)$.
The trivial monodromy representation locus will play a
slightly special role throughout the calculation. But, in the
final formula, no different treatment is required.

The space $H^1(C_0, f|_{C_0}^*T_{\proj^1[a]}(-\infty))$ yields
the vector bundle
$$\mathbb{B}=(\bE^{U})^\vee$$ over $\overline{\M}_{g,\gamma-\mu}
(\B\mathbb{Z}_a)$ whose rank  may be calculated by
the orbifold Riemann-Roch formula.
Over the component of the fixed locus where the
monodromy representation $\pi_1^{orb}(C_0)\to \mathbb{Z}_a$ is trivial,
the rank of $\mathbb{B}$ is $g$. Otherwise, the rank is
\begin{equation}\label{jjo}
r_\mathbb{B}=g-1+\sum_{i=1}^n \frac{\gamma_i}{a}+
\sum_{\mu_j\neq 0\ \text{mod}\ a}\left(1-\<\frac{\mu_j}{a}\>\right).
\end{equation}

The $H^1-H^0$ contribution from the collapsed component
to the localization formula is
\begin{equation}\label{vew}
\sum_{i=0}^{r_{\mathbb{B}}}\left(\frac{t}{a}\right)^{r_{\mathbb{B}}-i}
c_i(\mathbb{B}) =
\sum_{i= 0}^{r_{\mathbb{B}}} \left(\frac{t}{a}\right)^{r_{\mathbb{B}}-i}
 (-1)^i\lambda^{U}_i.
\end{equation}
For the component where the monodromy representation is trivial,
an additional factor of $\frac{a}{t}$ must be inserted in \eqref{vew}.

Next consider the $H^1-H^0$
contribution from the
$\com^*$-fixed Galois covers.
 Since
$$\text{deg}(f|_{C_j}^*T_{\proj^1[a]}(-\infty))=\frac{\mu_j}{a},$$
 we have
$$H^k(C_j, f|_{C_j}^*T_{\proj^1[a]}(-\infty))=
H^k\left(\proj^1,\sO_{\proj^1}\left(\left\lfloor
\frac{\mu_j}{a}\right\rfloor\right)\right).$$
The $H^0$ weights
are $$\frac{t}{\mu_j},2\frac{t}{\mu_j},...
,\left\lfloor\frac{\mu_j}{a}\right\rfloor\frac{t}{\mu_j},$$
where the weight 0 is omitted.{\footnote{The 0 weight
is from reparameterization of the domain $C_j$ and is
not in the virtual normal bundle.}}
The group $H^1$ vanishes.
The $H^1-H^0$ contribution is
\begin{equation*}
t^{-\left\lfloor\frac{\mu_j}{a}\right\rfloor}
\frac{\mu_j^{\left\lfloor\frac{\mu_j}{a}\right\rfloor}}
{\left\lfloor\frac{\mu_j}{a}\right\rfloor!}\ .
\end{equation*}

Finally, consider the $H^1-H^0$ contribution
from the nodal point $q_j$.
If $\mu_j\neq 0
\,(\text{mod }a)$, then $q_j$ is a stack point
and
$$H^0(q_j, f^*T_{\proj^1[a]}(-\infty)|_{q_j})=0$$
 as there is no invariant section.
If $\mu_j = 0 \,(\text{mod }a)$ then
$H^0(q_j, f^*T_{\proj^1[a]}(-\infty)|_{q_j})$ is 1-dimensional and
contributes a factor $\frac{t}{a}$. Certainly, $H^1$ vanishes here
for dimension
reasons.

The contribution from smoothing the node $q_j$ is
the tensor product of the tangent lines of the two branches
incident to $q_j$,
\begin{equation*}
e(N_j)=\frac{1}{b_j}\left(-\bar{\psi}_j+\frac{t}{\mu_j}\right).
\end{equation*}

After putting the component calculations together in \eqref{nnrrmm},
we obtain the following expression for
for $1/e(\text{Norm}^{vir})$:
$$\left(\sum_{i= 0}^{r_{\mathbb{B}}}
\left(\frac{t}{a}\right)^{r_\mathbb{B}-i}
(-1)^i \lambda_i^{U}\right)\cdot
\prod_{j=1}^\ell
\left(
t^{-\left\lfloor\frac{\mu_j}{a}\right\rfloor}
\frac{\mu_j^{\left\lfloor\frac{\mu_j}{a}\right\rfloor}}
{\left\lfloor\frac{\mu_j}{a}\right\rfloor!}
\frac{1}{\frac{1}{b_j}\left(-\bar{\psi}_j+\frac{t}{\mu_j}\right)}\right)
\cdot \prod_{j=1}^\ell
\left(\frac{t}{a}\right)^{\delta_{0,\<\frac{\mu_j}{a}\>}}.$$
Regrouping of terms yields
\begin{equation}\label{monn}
\frac{\prod_{j=1}^\ell
b_j\mu_j}{a^{r_{\mathbb{B}}+\sum_{j=1}^{\ell}
\delta_{0,\<\frac{\mu_j}{a}\>}}}
\left(\prod_{j=1}^\ell
\frac{\mu_j^{\left\lfloor\frac{\mu_j}{a}\right\rfloor}}
{\left\lfloor\frac{\mu_j}{a}\right\rfloor!}\right)
\left(\sum_{i=0}^{r_{\mathbb{B}}} t^{r_{\mathbb{B}}-i} (-a)^i
\lambda_i^{U}\right)
\cdot t^{-\sum_{j=1}^\ell
 \left\lfloor\frac{\mu_j}{a}\right\rfloor}
\prod_{j=1}^\ell \frac{t^{\delta_{0,\<\frac{\mu_j}{a}\>}}}
{(t-\mu_j\bar{\psi}_j)}.
\end{equation}
For the component with trivial  monodromy representation,
a factor of $\frac{a}{t}$ must be inserted in the formulas
for $1/e(\text{Norm}^{vir})$.

\label{ggg}
\subsection{Proof of Theorem \ref{vvv}}
Putting the calculations of Section \ref{ggg}
 together and passing to the non-equivariant limit, we
obtain the following evaluation
\begin{multline*}
\int_{[\overline{\M}_{g,\gamma}(\proj^1[a],\mu)]^{vir}}
\text{br}_0^*(H^r)  =
\frac{r!}{|\text{Aut}(\mu)|}\frac{a^{\ell}}
{a^{r_{\mathbb{B}}+
\sum_{j=1}^{\ell}\delta_{0,\<\frac{\mu_j}{a}\>}}}
\prod_{j=1}^{\ell}
\frac{\mu_j^{\left\lfloor\frac{\mu_j}{a}\right\rfloor}}
{\left\lfloor\frac{\mu_i}{a}\right\rfloor!}
\int_{\overline{\M}_{g,\gamma-\mu}(\B \mathbb{Z}_a)}
\frac{\sum_{i= 0}^\infty (-a)^i \lambda_i^{U}}{\prod_{j=1}^\ell
 (1-\mu_j\bar{\psi}_j)}.
\end{multline*}
On the right side, we have included the fundamental class
factors
$$\prod_{j=1}^\ell \frac{1}{\mu_j}$$
of the moduli spaces $P_j$.
For the component with trivial  monodromy representation,
a factor of $a$ must be inserted in the formula.

We can simplify the integral evaluation by using the calculation
\eqref{jjo} of $r_{\mathbb{B}}$,
\begin{equation*}
\begin{split}
&\quad r_{\mathbb{B}}+\sum_{i=1}^{\ell}\delta_{0,\<\frac{\mu_j}{a}\>}-\ell\\
&=g-1+\sum_{i=1}^n \frac{\gamma_i}{a}
+\sum_{\mu_j\neq 0\ \text{mod}\ a}\left(1-\<\frac{\mu_j}{a}\>\right)
+\left(\sum_{\mu_j = 0\ \text{mod}\ a}1\right)-\ell\\
&=g-1+\sum_{i=1}^n \frac{\gamma_i}{a} -\sum_{j=1}^{\ell}\<\frac{\mu_j}{a}\>.
\end{split}
\end{equation*}
The above calculation is not valid for the component with
trivial monodromy since $r_{\mathbb{B}}=g$ not $g-1$.
The discrepancy is exactly fixed by the extra factor $a$ required
for the
trivial monodromy case. We conclude
\begin{multline}\label{ffrrq}
\int_{[\overline{\M}_{g,\gamma}(\proj^1[a],\mu)]^{vir}}
\text{br}_0^*(H^r)  = \\
\frac{r!}{|\text{Aut}(\mu)|}
a^{1-g- \sum_{i=1}^n \frac{\gamma_i}{a} +\sum_{j=1}^\ell
\< \frac{\mu_j}{a} \>}
\prod_{j=1}^{\ell}
\frac{\mu_j^{\left\lfloor\frac{\mu_j}{a}\right\rfloor}}
{\left\lfloor\frac{\mu_i}{a}\right\rfloor!}
\int_{\overline{\M}_{g,\gamma-\mu}(\B \mathbb{Z}_a)}
\frac{\sum_{i=0}^\infty (-a)^i \lambda_i^{U}}{\prod_{j=1}^\ell
 (1-\mu_j\bar{\psi}_j)}.
\end{multline}
holds uniformly.
Theorem \ref{vvv} is then obtained from Lemmas \ref{ozoz} and
\ref{gew}. $\Box$

In degenerate cases, unstable integrals may appear on the right
side of the formula in Theorem \ref{vvv}.
The unstable integrals come in two forms
and are defined by the localization contributions:
$$
\int_{\overline{\M}_{0,(0)}(\B \mathbb{Z}_a)}
\frac{
\sum_{i\geq 0} (-a)^i \lambda_i^{U}
}
{(1-x \bar{\psi}_1)} = \frac{1}{a} \cdot \frac{1}
{x^2},$$

$$
\int_{\overline{\M}_{0,(m,-m)}(\B \mathbb{Z}_a)}
\frac{
\sum_{i\geq 0} (-a)^i \lambda_i^{U}
}
{(1-x \bar{\psi}_1)(1-y\bar{\psi}_2)} = \frac{1}{a} \cdot \frac{1}
{x+y}.$$
With the above definitions, Theorem \ref{vvv} holds in all
cases.
\label{bgt}

The disconnected formula \eqref{dderr} follows easily from the
connected case by the usual combinatorics of distributing
ramification points to the components of Hurwitz covers.

\subsection{Proof of Theorem \ref{vvvv}}
Suppose $\gamma$ satisfies the parity and strong negativity condition
with respect to $\mu$. 
Since
$$\delta= d-n - \frac{d- \sum_{i=1}^n \gamma_i}{a} <0,$$
the virtual dimension $r$ of $\overline{\M}_{g,\gamma}(\proj^1[a],\mu)$
is greater than $2g-2+d+\ell$. As a consequence, we immediately
obtain the vanishing
\begin{equation} \label{kkwq}
\int_{[\overline{\M}_{g,\gamma}(\proj^1[a],\mu)]^{vir}}
\text{br}^*(H^r) = 0
\end{equation}
since $H^r=0\in H^*(\text{Sym}^{2g-2+d+\ell}(\proj^1),\mathbb{Q})$.

We may nevertheless calculate \eqref{kkwq} by localization with the
lift $$H^r = (2g-2+d+\ell)[0]\cdot t^{-\delta}$$
which does {\em not} vanish equivariantly.
The analysis is identical to the calculations of Sections \ref{flo}-\ref{bgt}.
We find the integral \eqref{kkwq} is proportional (with nonzero factor)
to
\begin{equation*}
\int_{\overline{\M}_{g,\gamma-\mu}(\B \mathbb{Z}_a)}
\frac{\sum_{i=0}^\infty
(-a)^i \lambda_i^{U}}{\prod_{j=1}^\ell (1-\mu_j \bar{\psi}_j)},
\end{equation*}
and therefore conclude the vanishing. 

Assume now strong negativity does not hold,
but $\gamma$ satisfies the parity, negativity, and
boundedness condition.
By the proof of Lemma \ref{HHH}, using the boundedness condition,
the maps $$f: C \rightarrow \proj^1[a]$$
which satisfy $[0]\notin \text{br}_0(f)$ have no contraction over 0 and coarse
profile determined by $\gamma$. By the negativity condition, no such maps
exists. Hence, $[0]$ is always in $\text{br}_0(f)$. Therefore,
\begin{equation*}
\int_{[\overline{\M}_{g,\gamma}(\proj^1[a],\mu)]^{vir}}
\text{br}_0^*(H^r) = 0
\end{equation*}
and we conclude as above.$\Box$

\section{Examples}

\subsection{$\mathbb{Z}_2$ example}
The Hodge bundle $\mathbb{E}^{U}$ has a very simple interpretation
in the $\mathbb{Z}_2$ case.
Let
$$ \C \rightarrow
\overline{\M}_{g,\gamma}(\B \mathbb{Z}_2), \ \ \D\rightarrow \C $$
be the universal domain curve and
the universal $\mathbb{Z}_2$-bundle.
Let
$$
\epsilon: \overline{\M}_{g,\gamma}(\B \mathbb{Z}_2) \rightarrow
\overline{\M}_g, \ \
\tilde{\epsilon}: \overline{\M}_{g,\gamma}(\B \mathbb{Z}_2)
\rightarrow \overline{\M}_{g-1+\frac{n}{2}}$$
be the maps to moduli obtained from $\C$ and $\D$ respectively.
The  exact sequence
$$0 \rightarrow \epsilon^*(\mathbb{E}_g) \rightarrow
\tilde{\epsilon}^*(\mathbb{E}_{g-1+\frac{n}{2}}) \rightarrow
\mathbb{E}^{U} \rightarrow 0.$$
exhibits $\mathbb{E}^U$ as the $K$-theoretic difference
of the pulled-back Hodge bundles.
If $g=0$, then the situation{\footnote{The
map $\epsilon$ is not well-defined here for stability reasons.}} is even
simpler,
\begin{equation}\label{frex}
\mathbb{E}^U \stackrel{\sim}{=} \tilde{\epsilon}^*
(\mathbb{E}_{g-1+\frac{n}{2}}).
\end{equation}

Consider the case of
Theorem \ref{vvv} where $g=0$, $\gamma=(1,1)$,
and $\mu=(1,1)$. The statement is
$$H_0((1,1),(1,1)) = \frac{2}{2!2!} 2^1
\int_{\overline{\M}_{0,(1,1,1,1)}(\B \mathbb{Z}_2)} \frac{1-2\lambda^U_1}
{(1-\bar{\psi}_1)(1-\bar{\psi}_2)}.$$
The double Hurwitz number on the left is $\frac{1}{2}$.
Expansion of the right side yields:
\begin{eqnarray*}
\int_{\overline{\M}_{0,(1,1,1,1)}(\B \mathbb{Z}_2)}
\frac{1-2\lambda_1^U}{(1-\bar{\psi}_1)(1-\bar{\psi}_2)}
& = &  \frac{1}{2} \int_{\overline{\M}_{0,4}} \frac{1}{(1-\psi_1)(1-\psi_2)}
-2 \int_{\overline{\M}_{0,(1,1,1,1)}(\B \mathbb{Z}_2)} \lambda^U_1 \\
& = & 1 - 2
\int_{\overline{\M}_{0,(1,1,1,1)}(\B \mathbb{Z}_2)} \lambda^U_1.
\end{eqnarray*}
To evaluate the last integral, we note the map
$$\tilde{\epsilon}:\overline{\M}_{0,(1,1,1,1)}(\B \mathbb{Z}_2)
\rightarrow \overline{\M}_{1,1},$$ where the first branch point
is selected for the marking on the elliptic curve,
is of degree 6. Moreover, $\lambda_1^U$ is the pull-back
of $\lambda_1$ under $\tilde{\epsilon}$ by \eqref{frex}. Hence,
$$
 1 - 2
\int_{\overline{\M}_{0,(1,1,1,1)}(\B \mathbb{Z}_2)}
\lambda^U_1  = 1 - 2 \cdot 6 \cdot \frac{1}{24} = \frac{1}{2}.$$

\subsection{Vanishing example}
The simplest example of the vanishing of Theorem 2 occurs for
$\mathbb{Z}_2$. Let $g=0$,
$$\gamma=(\underbrace{1,\ldots,1}_{n})$$
and $\mu=(1)$. By the parity condition, $n$ must be odd.
Boundedness holds.
For the negativity condition, we require $n\geq 2$. 
By Theorem 2 (i),
$$\int_{\overline{\M}_{0,\gamma-\mu}(\B \mathbb{Z}_2)} \frac{\sum_{i\geq 0}
(-2)^i\lambda^U_i}{1-\bar{\psi}_1}$$
vanishes for all odd $n\geq 3$.

We now use the identification of $\lambda_i^U$ with the Chern classes
of the Hodge bundle $\tilde{\epsilon}^*(\mathbb{E}_{\frac{n-1}{2}})$
whose fiber over
$$f:[D/\mathbb{Z}_2] \rightarrow \B \mathbb{Z}_2$$
is simply given by the space of differential
forms on the genus $\frac{n-1}{2}$ curve $D$. 
The Chern roots of $\tilde{\epsilon}^*(\mathbb{E}_{\frac{n-1}{2}})$
can be identified by the vanishing sequence
at a Weierstrass point of $D$. The Weierstrass point can be chosen to
lie above the marking corresponding to the single
part of $\mu$. The Chern roots of 
$\tilde{\epsilon}^*(\mathbb{E}_{\frac{n-1}{2}})$
 are then
$L, 3L, \ldots, (n-2)L$ where $L$ is the Chern class of the
cotangent line of
the Weierstrass point. 
The class $L$ on $\overline{\M}_{0,\gamma-\mu}(\B \mathbb{Z}_2)$
is $\frac{1}{2}\bar{\psi}_1$.
Expanding the Chern roots, we find
\begin{eqnarray*}
\int_{\overline{\M}_{0,\gamma-\mu}(\B \mathbb{Z}_2)} \frac{\sum_{i\geq 0}
(-2)^i\lambda^U_i}{1-\bar{\psi}_1}
& = &   
\int_{\overline{\M}_{0,\gamma-\mu}(\B \mathbb{Z}_2) } 
\frac{\prod_{i=1}^{\frac{n-1}{2}} (1-(2i-1)\bar{\psi}_1)}
{(1-\bar{\psi}_1)} \\
& = &  
\int_{   \overline{\M}_{0,\gamma-\mu}(\B \mathbb{Z}_2)       } 
{\prod_{i=2}^{\frac{n-1}{2}} (1-(2i-1)\bar{\psi}_1)} \\
& = & 0,
\end{eqnarray*}
where the last integral vanishes for dimension reasons.

\subsection{$\mathbb{Z}_\infty$ example}
An interesting feature of Theorem \ref{vvv} is the possibility of studying the behavior
for large $a$. Let $\gamma=(\gamma_1,\ldots,\gamma_n)$ determine a partition of $d$,
$$d= \sum_{i=1}^n \gamma_i.$$
Let $\mu=(d)$ consist of a single part.
For $a>d$, the rank of the Hodge bundle 
$$\mathbb{E}^U \rightarrow
\overline{\M}_{0,\gamma-\mu}(\B\mathbb{Z}_a)$$ is 0 by  \eqref{jjo}.
Since the parity, non-negativity, and boundedness conditions hold
for $a>d$, we may apply Theorem 1 to conclude
\begin{eqnarray*}
H_{0}(\gamma,(d))& = & \frac{(n-1)!}{|\text{Aut}(\gamma)|}\  a \int_{\overline{\M}_{0,\gamma-\mu}(\B \mathbb{Z}_a)}
\frac{1}{1-d\bar{\psi}_1} \\
& = & \frac{(n-1)!}{|\text{Aut}(\gamma)|}\  d^{n-2},
\end{eqnarray*}
which is a well-known formula for genus 0 double Hurwitz numbers.

\subsection{1-point series} \label{bnf}
If $\mu=(d)$ consists of a single part, the entire generating series
for double Hurwitz numbers has been computed 
\footnote{We write Theorem 3.1 of \cite{GJV} in terms of
$\sin$ instead of $\sinh$ and divide
by
 $|\text{Aut}(\nu)|$ since we do not mark ramifications in
our definition of Hurwitz numbers.}
in \cite{GJV}:
\begin{equation} \label{explicithurwitz}
\sum_{g\geq 0} t^{2g}(-1)^g H_g(\nu, (d))
=\frac{r!\ d^{r-1}}{|\text{Aut}(\nu)|}
\prod_{k\geq 1}\left
(\frac{\text{sin}(kt/2)}{kt/2}\right)^{m_k(\nu)-\delta_{k,1}},
\end{equation}
where $r=r_g(\nu,(d))$ and
 $m_k(\nu)$ is the number of times $k$ appears as a part of $\nu$.
Single part double Hurwitz numbers are considerably simpler 
because such covers are automatically connected and the only characters 
with nonzero evaluation on the $d$-cycle are exterior powers of the 
standard $(d-1)$-dimensional representation.

 Let $\gamma=(\gamma_1,\ldots,\gamma_n)$ be a vector 
of nontrivial
elements of $\mathbb{Z}_a$ satisfying the boundedness condition.  
We will consider degrees $d$ for which the parity and non-negativity
 conditions 
are satisfied.  Then,
 $$d-\sum_{i=1}^n \gamma_i=ab$$
for an integer $b\geq 0$.
Consider the generating series
$$
F_{\gamma}(t,z)=\sum_{g= 0}^\infty
\sum_{l=-\infty}^g t^{2g}z^l
\int_{\overline{\M}_{g,\gamma-(d)}(\B \mathbb{Z}_a)}
\bar{\psi}_0^{2g-2+\ell(\gamma)+l} \lambda^U_{g-l}
$$
where 
$\bar{\psi}_0$ is the class corresponding to the point with monodromy $-d$.

The double Hurwitz number formula of 
Theorem \ref{vvv} is
\begin{eqnarray*} H_g(\gamma_+, (d))&=& \frac{r!}{|\text{Aut}(\gamma)|}a^{1-g-
\sum_{i=1}^n \frac{\gamma_i}{a}+  \< \frac{d}{a} \>}
\frac{d^{\left\lfloor\frac{d}{a}\right\rfloor}}{\left\lfloor\frac{d}{a}\right\rfloor !}
\sum_{l=-\infty}^g d^{r-b-1+l}(-a)^{g-l}
\int_{\overline{\M}_{g,\gamma-(d)}(\B\mathbb{Z}_a)}\bar{\psi}_0^{r-b-1+l} \lambda^U_{g-l}         \\
&=&
(-1)^g\frac{ad^{r-1}r!\left(\frac{d}{a}\right)^{\left\lfloor\frac{\sum \gamma_i}{a}\right\rfloor}}
{|\text{Aut}(\gamma)|\left(b+\left\lfloor\frac{\sum \gamma_i}{a}\right\rfloor\right)!} \sum_{l=-\infty}^g
\left(\frac{-d}{a}\right)^l
\int_{\overline{\M}_{g,\gamma-(d)}(\B\mathbb{Z}_a)}\bar{\psi}_0^{r-b-1+l} \lambda^U_{g-l}
\end{eqnarray*}
or, equivalently,
$$
\sum_{g\geq 0} (-1)^gt^{2g}H_g(\gamma_+, (d))
=\frac{ad^{r-1}r!}{|\text{Aut}(\gamma)|\left(b+\left\lfloor\frac{\sum \gamma_i}{a}\right\rfloor\right)!}
\left(\frac{d}{a}\right)^{\left\lfloor\frac{\sum\gamma_i}{a}\right\rfloor}
F_{\gamma}(t, -d/a)
$$
where $r=r_g(\gamma_+,(d))$.
After combining with \eqref{explicithurwitz}, we obtain
\begin{equation} \label{fnice}
F_{\gamma}(t, -d/a)=
\frac{1}{a} \frac{\left(b+\left\lfloor\frac{\sum \gamma_i}{a}\right\rfloor \right)!}{b!}
\left(\frac{a}{d}\right)^{\left\lfloor\frac{\sum \gamma_i}{a}\right\rfloor}
\prod_{k\geq 1}
\left(\frac{\sin(kt/2)}{kt/2}\right)^{m_k(\gamma_+)-\delta_{k,1}}.
\end{equation}
for $b\geq 0$.

\begin{thm}  \label{vrt}
$F_{\gamma}(t, z)$ equals
\begin{equation*}
\frac{1}{a} \frac{\left(-z - \sum \frac{\gamma_i}{a} +
\sum\left\lfloor\frac{\sum \gamma_i}{a}\right\rfloor \right)!}
{\left(-z - \sum\frac{\gamma_i}{a}\right)!}
(-z)^{-\left\lfloor\frac{\sum \gamma_i}{a}\right\rfloor}
\left(\frac{\sin(at/2)}{at/2}\right)^{-z - \sum \frac{\gamma_i}{a}}
\prod_{k\geq 1}
\left(\frac{\sin(kt/2)}{kt/2}\right)^{m_k(\gamma)-\delta_{k,1}}.
\end{equation*}
\end{thm}

\begin{proof}
Using the standard polynomial expansion 
$$\frac{\left(-z - \sum \frac{\gamma_i}{a} +
\sum\left\lfloor\frac{\sum \gamma_i}{a}\right\rfloor \right)!}
{\left(-z - \sum\frac{\gamma_i}{a}\right)!} =
\left(-z - \sum \frac{\gamma_i}{a} +
\sum\left\lfloor\frac{\sum \gamma_i}{a}\right\rfloor \right) \ldots
\left(-z - \sum \frac{\gamma_i}{a} + 1 \right),
$$
we see
the $t^{2g}$ coefficients of both sides of Theorem \ref{vrt}
are Laurent polynomials
in $z$.
Equation \eqref{fnice} shows Theorem \ref{vrt} holds for
all evaluations of the form
$z= -d/a$ where 
$$d-\sum_{i=1}^n \gamma_i=ab$$
and $b$ is a non-negative integer.
Since there are infinitely many such evaluations, the coefficient Laurent
polynomials in $z$ must be equal for all $t^{2g}$.
\end{proof}

If we specialize Theorem 4 to the case 
where  $\gamma=\emptyset$, we obtain
\begin{equation} \label{onemarkedpoint}
\frac{1}{a}
+\sum_{g>0}\sum_{l=0}^g t^{2g}z^l\int_{\overline{\M}_{g,1}(\B\mathbb{Z}_a)}\bar{\psi}_1^{2g-2+l}\lambda^U_{g-l}
=\frac{1}{a} 
\left(\frac{at/2}{\sin(at/2)}\right)^{z} \frac{t/2}{\sin(t/2)}
\end{equation}
If $\gamma=\emptyset$ and $a=1$ we recover 
\begin{equation} \label{FPequation}
1+\sum_{g>0}\sum_{l=0}^g t^{2g} z^l\int_{\overline{\M}_{g,1}}\psi_1^{2g-2+l}\lambda_{g-l}
=\left(\frac{t/2}{\sin(t/2)}\right)^{z+1}
\end{equation}
first calculated in \cite{FPH}.

In (\ref{onemarkedpoint}), the term $\lambda^U_g$ vanishes for dimensional reasons except over the trivial monodromy component, 
where it agrees with the usual $\lambda_g$.  
Indeed, setting $z=0$ in (\ref{onemarkedpoint}) yields
$$\frac{1}{a}+\sum_{g>0}t^{2g}
\int_{\overline{\M}_{g,1}(\B\mathbb{Z}_a)}\psi_1^{2g-2}\lambda^U_g=\frac{1}{a}\frac{t/2}{\sin(t/2)}$$
which is the expected contribution from (\ref{FPequation}) with a factor of $1/a$ coming from the automorphisms.

\section{Abelian groups}
\subsection{Pull-back}
  For an abelian group $G$ and irreducible representation $R$,
 recall the sequence (\ref{exactsequence}),
$$0 \rightarrow K \rightarrow G \stackrel{\phi^R}{\rightarrow}
 \text{Im}(\phi^R) \stackrel{\sim}{=} \mathbb{Z}_a \rightarrow0.$$
By construction $R\stackrel{\sim}{=}\phi^{R*}(U)$.
The homomorphism $\phi^R$ induces a canonical map
$$\rho: \overline{\M}_{g,\gamma}(\B G) \rightarrow
\overline{\M}_{g,\phi^R(\gamma)}(\B \mathbb{Z}_a)$$
by sending a principal $G$-bundle to its quotient by $K$.

\begin{lemma} \label{pullbacklemma}
$\bE^R\stackrel{\sim}{=}\rho^*(\bE^U)$.
\end{lemma}
\begin{proof}
Recall $\bE\rightarrow \overline{\M}_{g,n}(\B H)$ 
is the bundle whose fiber over 
$$[f]:[D/H]\to \B H\in \overline{\M}_{g,n}(\B H)$$
 is  $H^0(D, \omega_D)$.
The latter 
can be understood as the space of 
 1-forms $\alpha$ on the 
normalization $\tilde{D}$ of $D$
with possible simple poles with opposite residues at the two 
preimages of each node $q_i$.

Let $\tilde{\rho}$ be the map between the universal principal $G$- 
and $\mathbb{Z}_a$-curves that induces $\rho$.  
We obtain
  $$d\tilde{\rho}:\rho^*(\bE)\to \bE$$
by pulling-back differential forms.
An easy verification shows $\tilde{\rho}$ is well-defined even at
points in the moduli space $\overline{\M}_{g,\gamma}(\B G)$
for which the
$G$-curve is nodal.  

The map $d\tilde{\rho}$ is injective on each fiber since 
the pull-back of a nonzero differential form by a finite surjective 
map is nonzero.  
Certainly $d\tilde{\rho}$ carries the subbundle 
$\rho^*(\bE^U)$ to the subbundle $\bE^R$.  
These bundles have the same dimension by the 
Riemann-Roch formula for twisted curves. Hence,
 $d\tilde{\rho}$ is an isomorphism.
\end{proof}

The map $\rho$ does not preserve the 
isotropy groups at the marked points.
However, since the classes $\bar{\psi}_i$ 
are pulled-back from $\overline{\M}_{g,n}$, 
$$\rho^*(\bar{\psi})=\bar{\psi}.$$   
By Lemma \ref{pullbacklemma}, we 
concluded the integrand in Theorem \ref{vvvvv} is exactly the 
integrand of Theorem $\ref{vvv}$ pulled-back via $\rho$.

\subsection{Degree}

The degree of $\rho$ is determined by the following result.
\begin{lemma} \label{degreelemma} We have
$$
\mathrm{deg}(\rho)=\left\{
\begin{array}{ll}
0 & \sum_i \gamma_i \neq 0 \\
|K|^{2g-1} & \sum_i \gamma_i =0
\end{array} \right. .
$$

\end{lemma}
\begin{proof}
Consider a nonsingular curve $[C,p_1,\dots, p_n]\in\overline{\M}_{g,n}$.  Let
$$ \Gamma=\pi_1(C \setminus \{p_1,\dots, p_n\})
=\left\langle \Gamma_i, A_j, B_j\Big|\prod_{i=1}^n 
\Gamma_i \prod_{j=1}^g[A_j, B_j]\right\rangle,$$
where $\Gamma_i$ is a loop around $p_i$ and 
the loops $A_j, B_j$ are the standard generators of $\pi_1(C)$.

The elements of $\overline{\M}_{g,\gamma}(\B G)$ lying above 
$[C, p_1,\dots p_n]$ are in bijective correspondence
with the homomorphisms{\footnote{
Composition in $\Gamma$ 
is written multiplicatively while composition in $G$ is additive.}}
 $\varphi:\Gamma\to G$ with
\begin{equation} \label{phiequation}
\varphi(\Gamma_i)=\gamma_i.
\end{equation}
Since $G$ is abelian, $\varphi( [A_j, B_j] )=0$. Hence,
the parity condition
\begin{equation} \label{nonempty}
\sum_{i=1}^n \gamma_i=0
\end{equation}
must be satisfied for $\overline{\M}_{g,\gamma}(\B G)$ to be nonempty.

If the parity condition holds, 
then the images of $A_j$ and $B_j$ are completely unconstrained.
There are $|G|^{2g}$ homomorphisms $\phi$ satisfying (\ref{phiequation}).  
Stated in terms of homomorphisms, the map $\rho$ corresponds to 
composing $\varphi:\Gamma\to G$ with $\phi^R:G\to\mathbb{Z}_a$.  
Since there are $|K|$ elements of $G$ in the preimage of any element 
of $\mathbb{Z}_a$, there are $|K|^{2g}$ elements in a generic fiber of $\rho$. 
Since $G$ is abelian, a cover 
in $\mathcal{\M}_{g,\gamma}(\B G)$ has automorphism group $G$. 
A cover in the image of $\rho$ only has automorphism group $\mathbb{Z}_a$. 
Thus, the degree of $\rho$ is $|K|^{2g-1}$.
\end{proof}

Although $\overline{\M}_{g,\phi^R(\gamma)}(\B \mathbb{Z}_a)$ may
 have several components, 
Lemma \ref{degreelemma} implies 
the degree of $\rho$ is the same over each component.  
In the nonabelian case, the situation is
much more complicated.
For example, let $\eta$ be the conjugacy class of a $3$-cycle in 
$\Sigma_3$, let
$$s:\Sigma_3\to\mathbb{Z}_2$$
be the sign representation, and let
$$\rho:\overline{\M}_{1,\eta}
(\B\Sigma_3)\to\overline{\M}_{1, 0}(\B \mathbb{Z}_2)$$
be the map induced by $s$.   
The space $\overline{\M}_{1, 0}(\B \mathbb{Z}_2)$ consists of 
two components: one with 
trivial monodromy, and one with nontrivial monodromy.  
There are covers in $\overline{\M}_{1,\eta}(\B\Sigma_3)$ lying above the 
nontrivial monodromy component. If
$t_1\neq t_2\in\Sigma_3$
are two transpositions, then $[t_1,t_2]$ is a 3-cycle.
On the other hand, there are no elements of 
$\overline{\M}_{1,\eta}(\B\Sigma_3)$ 
lying above the trivial monodromy component. All the monodromy in such a 
cover would lie in the abelian group 
$\mathbb{Z}_3=\ker(s)$, and there are no such covers with nontrivial monodromy about the one marked point by (\ref{nonempty}).  
As the formula in Theorem \ref{vvv} considers all components of 
$\overline{\M}_{g,\phi^R(\gamma)}(\B\mathbb{Z}_a)$ at once, 
a more nuanced approach would be required to understand 
Hurwitz-Hodge integrals for nonabelian groups, 
even for 1-dimensional representations.

In the disconnected case $\rho: \overline{\M}^\bullet
_{g,\gamma}(\B G) \rightarrow
\overline{\M}^\bullet_{g,\phi^R(\gamma)}(\B \mathbb{Z}_a)$,
Lemma \ref{degreelemma} has a few minor complications:
\begin{enumerate}
 \item[(i)] The monodromy condition $\sum_i \gamma_i=0\in G$ cannot be 
checked globally, but must be verified separately 
on each domain component.

\item[(ii)] The number of components matters. For disconnected curves 
with $h$ components, each of which satisfies the monodromy requirements, 
the degree of $\rho$ is $|K|^{2g-2+h}$.
\end{enumerate}
When $\rho$ is nonzero, the degree $|K|^{2g-2+h}$
 is independent of $G$ and the monodromy conditions (\ref{nonempty}). 
The only role these conditions play is to determine when the degree is nonzero.

\subsection{Wreath Hurwitz numbers}

The wreath product $K_d$ is defined by
$$
K_d=\{(k,\sigma)\ | \ k=(k_1,\dots, k_d)\in K^d, \sigma\in \Sigma_d\},
$$
$$
(k,\sigma)(k^\prime, \sigma^\prime)=(k+\sigma(k^\prime),\sigma \sigma^\prime).
$$
Conjugacy classes of $K_d$ are determined by their cycle types \cite{M}.
Since $K$ is abelian,
 for each $m$-cycle $(i_1i_2\cdots i_m)$ of $\sigma$, the element 
$k_{i_m}+k_{i_{m-1}}+\cdots +k_{i_1}$ is well-defined.
The resulting $\text{Conj}(K)$-wieghted partition of $d$ is the
called the {\em cycle type} of $(k,\sigma)$.
Two elements of $K_d$  are conjugate exactly 
when they have the same cycle type.

We index the conjugacy classes of $K_d$ by 
$\text{Conj}(K)$-weighted partitions of $d$.
Let
$$ \overline{\nu}=\{(\nu_1,\iota_1), \dots, (\nu_{\ell(\nu)}, \iota_{\ell(\mu)})\}, $$
$$ \overline{\mu}=\{(\mu_1,\kappa_1), \dots, (\mu_{\ell(\mu)}, \kappa_{\ell(\mu)})\} $$
be two such partitions.
Let $\nu^*$ be the partition with parts $\nu_j$ with a partial labelling
given by $\iota_j$. Then
$$\text{Aut}(\nu^*)= \text{Aut}(\overline{\nu}).$$ 
The Hurwitz number $H_g(\nu^*,\mu^*)$ counts cover with
the additional labelling data,
$$H_g(\nu^*,\mu^*) =   \frac{|\text{Aut}(\nu)|}{ |\text{Aut}(\nu^*)|}
\frac{|\text{Aut}(\mu)|}{ |\text{Aut}(\mu^*)|}
        H_g(\nu,\mu).$$

\begin{lemma} \label{alternatecount}
$H_{g, K}(\overline{\nu}, \overline{\mu})$ 
is the count of the covers 
$\pi:C\to \proj^1$ enumerated by $H_g(\nu^*, \mu^*)$
with multiplicity $m_\pi$.
The multiplicity $m_\pi$ is the 
 automorphism-weighted count of principal $K$-bundles on 
$C\setminus \pi^{-1}(\{0,\infty\})$ with monodromy $\iota_i$ at
 $p_i\in\pi^{-1}(0)$ and $\kappa_j$ at $q_j\in\pi^{-1}(\infty)$.
\end{lemma}

\begin{proof}
Let $\pi^\prime:D\to\mathbb{P}^1$ be a cover counted by 
$H_{g,K}(\overline{\nu},\overline{\mu})$. By definition,
$\pi'$ is a $d|K|$-fold cover of $\mathbb{P}^1$ with 
monodromies $\overline{\nu}, \overline{\mu}$ and $\overline{\tau}$ 
over $0,\infty$ and the points of $U_r$  respectively.

Each such cover has an associated cover $\pi: C\to\proj^1$ counted by 
$H_g(\nu^*, \mu^*)$.  Algebraically, the cover is obtained by the 
forgetful map from $K_d\to \Sigma_d$.
 Geometrically, the cover is obtained by taking the quotient of $D$ by 
the diagonal subgroup $K\subset K_d$.
There is a natural map $f:D\to C$.  Away from the preimages of $0,\infty$ and $U_r$, the map $f$ is a principal $K$-bundle.

Consider the point $p_i\in\pi^{-1}(0)$ corresponding to a cycle $\nu_i$ 
which is labelled with $\iota_i\in K$.  A small loop winding once around $p_i$ 
on $C$ has an image that winds $\nu_i$ times around $0$. 
 But we know that the monodromy for $\pi^\prime:D\to\mathbb{P}^1$ around 
$0$ is given by $\overline{\nu}$.  
By the definition of the cycle type,
the monodromy of $f$ around $p_i$ is $\iota_i$.  
An identical argument shows the monodromy at $q_i$ over $\infty$ is $\kappa_i$
 and the monodromy around all preimages of a point in $U_r$ is zero.

The above process is reversible. We start with
a $d$-fold cover $\pi^\prime:C\to \proj^1$ counted by 
$H_g(\nu^*, \mu^*)$ and a principal $K$-bundle 
$f:D\to C$ with monodromy $\iota_i$ around $p_i$ 
and $\kappa_i$ around $q_i$.
Then, the composition $\pi = \pi' \circ f$
is a cover counted by
$H_{g,K}(\overline{\nu},\overline{\mu})$.
\end{proof}

In other words, if
$\rho^\prime:\overline{\M}_{g,\iota\cup\kappa}(\B K)
\to\overline{\M}_{g,\ell(\lambda)+\ell(\mu)}$
is the natural map, then
$$H_{g,K}(\overline{\nu},\overline{\mu})
=\text{deg}(\rho^\prime) H_g(\nu^*,\mu^*).$$

\subsection{Proof of Theorem \ref{vvvvv}}
By Lemma \ref{pullbacklemma}, 
we can compute the integral in Theorem \ref{vvvvv} by computing 
the analogous Hurwitz-Hodge integral (appearing in Theorem \ref{vvv}) 
over $\overline{\M}_{g,-\mu}(\B \mathbb{Z}_a)$ 
and  multiplying by the degree of 
$$\rho:\overline{\M}_{g,-\overline{\mu}}(\B G)
\to\overline{\M}_{g,-\mu}(\B \mathbb{Z}_a).$$
On the other hand, by Lemma $\ref{alternatecount}$, we can 
calculate $H_{g,K}(\emptyset_+(k),\overline{\mu})$ 
by computing 
$H_{g}(\emptyset_+,\mu)$,
multiplying by the degree of 
$$\rho':\overline{\M}_{g,(-k)^{d/a}\cup\kappa}(\B K)\to 
\overline{\M}_{g,d/a+\ell(\mu)},$$
and correcting for the difference in the sizes of the automorphism
groups $\text{Aut}(\mu)$ and 
$$\text{Aut}(\overline{\mu})
=\text{Aut}(\mu^*).$$

Thus, to deduce Theorem \ref{vvvvv} from Theorem \ref{vvv}, 
we need only check that the degrees of $\rho$ and $\rho^\prime$ agree.  
By Lemma \ref{degreelemma}, the degrees agree when nonzero. The last step is
to
check the parity condition (\ref{nonempty}) is the same for 
$\rho$ and $\rho^\prime$.
For $\rho$, the parity condition is
$$0= \sum_{j=1}^\ell (-\overline{\mu})_j
=\sum_{j=1}^\ell ( \kappa_j-\mu_j x)=\sum_{j=1}^\ell \kappa_j-dx.$$
For $\rho^\prime$, the parity condition is
$$0=- \frac{d}{a}k+\sum_{j=1}^\ell \kappa_j.$$
Since $ax=k$, the conditions are equivalent. $\Box$

As in the faithful case, unstable integrals may appear on the right
side of the formula in Theorem \ref{vvvvv}.  These unstable terms are 
defined in a completely analogous manner, 
and extend Theorem \ref{vvvvv} to all contributions:

$$
\int_{\overline{\M}_{0,(0)}(\B G)}
\frac{
\sum_{i\geq 0} (-a)^i \lambda_i^{R}
}
{(1-x \bar{\psi}_1)} = \frac{1}{|G|} \cdot \frac{1}
{x^2},$$

$$
\int_{\overline{\M}_{0,(m,-m)}(\B G)}
\frac{
\sum_{i\geq 0} (-a)^i \lambda_i^{R}
}
{(1-x \bar{\psi}_1)(1-y\bar{\psi}_2)} = \frac{1}{|G|} \cdot \frac{1}
{x+y}.$$

Alternatively, using  a theory of stable maps relative to a stack divisor{\footnote{We avoid the foundational discussion of this theory.}}
at $\infty$, 
Theorem \ref{vvvvv} could be proven in a manner closely parallel to 
the proof of Theorem \ref{vvv}.


\begin{thebibliography}{123}
\bibitem{AGV} D. Abramovich, T. Graber, A. Vistoli, {\em Gromov-Witten theory for Deligne-Mumford stacks}, math/0603151.

\bibitem{AV} D. Abramovich and A. Vistoli, {\em Compactifying the space of stable maps}, JAMS {\bf 15} (2002), 27--75.

\bibitem{BGP} J. Bryan, T. Graber, and R. Pandharipande,
{\em The orbifold quantum cohomology of ${\mathbb{C}}^2/{\mathbb{Z}}_3$
and Hurwitz-Hodge integrals},
           J. Alg. Geom. {\bf 17} (2008), 1 -- 28.

\bibitem{C} C. Cadman, {\em Using stacks to impose tangency conditions on curves}, Amer. J. Math. {\bf 129} (2007), 405--427.


\bibitem{CR} W. Chen and Y. Ruan, {\em Orbifold Gromov-Witten theory} 
in {\em Orbifolds in mathematics and physics (Madison, WI, 2001)},
25--85, 
Contemp. Math. {\bf 310} (2002).


\bibitem{ELSV} T. Ekedahl, S. Lando, M. Shapiro, and A. Vainshtein, {\em Hurwitz numbers and intersections on moduli spaces of curves}, Invent. Math. 146 (2001), 297–-327.

\bibitem{FPH} C. Faber and R. Pandharipande, {\em Hodge integrals and
Gromov-Witten theory},  Invent. Math.  {\bf 139}  (2000), 173--199.


\bibitem{FPJ} C. Faber and R. Pandharipande,
{\em Relative maps and tautological
classes}, JEMS {\bf 7} (2005), 13--49.

\bibitem{FP} B. Fantechi, R. Pandharipande, 
{\em Stable maps and branch divisors}, Compositio Math. {\bf 130} (2002), 345--364.

\bibitem{GP} T. Graber and R. Pandharipande, {\em Localization of virtual classes}, 
Invent. Math. {\bf 135} (1999), 487--518.


\bibitem{GPC} T. Graber and R. Pandharipande, {\em Constructions of
nontautological classes on moduli spaces of curves}, Michigan Math J
{\bf 51} (2003), 93--109.

\bibitem{GV1} T. Graber and R. Vakil, {\em Hodge integrals and Hurwitz numbers
via virtual localization}, Compositio Math. {\bf 135} (2003), 25--36.

\bibitem{GV2} T. Graber and R. Vakil, {\em Relative virtual localization and vanishing of tautological classes on moduli spaces of curves}, Duke Math. J. 
{\bf 130} (2005), 1--37.

\bibitem{GJV} I.P. Goulden, D.M. Jackson, and R. Vakil, {\em Towards the
geometry of double Hurwitz numbers}, Advances in Mathematics {\bf 198} (2005),
43--92.



\bibitem{HM} J. Harris and D. Mumford, {\em On the Kodaira dimension
of the moduli space of curves}, Invent. Math. {\bf 67} (1982), 23--88. 




\bibitem{JK} T. Jarvis and T. Kimura, {\em Orbifold quantum cohomology of
the classifying space of a group} in {\em Orbifolds in mathematics and physics
(Madison, WI, 2001)}, 123--134, 
Contemp. Math. {\bf 310} (2002).

\bibitem{PJ} P. Johnson, in preparation.


\bibitem{KL}  M. Kazarian and S. Lando, {\em An algebro-geometric proof of
Witten's conjecture}, JAMS {\bf 20} (2007), 1079 -- 1089.


\bibitem{M} I. G. Macdonald, {\em Symmetric Functions and Hall Polynomials}, 2nd ed., Oxford
University Press, 1995.



\bibitem{Ok} A. Okounkov, {\em Toda equations for Hurwitz numbers},
Math. Res. Letters {\bf 7} (2000), 447-453.

\bibitem{OP1} A.~Okounkov and R.~Pandharipande, {\em Gromov-Witten theory,
Hurwitz numbers, and matrix models}, math/0102017.

\bibitem{OP2}  A. Okounkov and R. Pandharipande, {\em Gromov-Witten theory,
Hurwitz numbers, and completed cycles}, Ann. of Math {\bf{163}} (2006),
517 -- 560.

\bibitem{OP3} A. Okounkov and R. Pandharipande,
{\em The equivariant
Gromov-Witten theory of ${\mathbf P}^1$}, Ann. of Math {\bf{163}} (2006), 
561 -- 605.

\bibitem{P} R. Pandharipande, {\em The Toda equation and the Gromov-Witten
theory of the Riemann sphere}, Lett. Math. Phys. {\bf 53} (2000), 59 -- 74.


\bibitem{QW} Z. Qin and W. Wang, {\em Hilbert schemes of points on the minimal
resolution and soliton equation} in {\em Lie algebras, 
vertex operator algebras
and their applications}, Y.-Z. Huang and K. Misra (eds.), 
435--462,
Contemp. Math.
{\bf 442} (2007).

\end{thebibliography}
\end{document}